\documentclass[11pt]{article}
\usepackage[utf8]{inputenc}
\usepackage{lmodern}
\usepackage{subfiles}
\usepackage{enumitem}
\setenumerate{topsep=6pt,ref={\normalfont(\arabic*)},label={\normalfont(\arabic*)}, itemsep=0pt} % or \upshape
        
\usepackage{amsfonts}
\usepackage{amsthm}
\usepackage{amsmath}
\usepackage{amssymb}
\usepackage{amscd}
\usepackage{mathrsfs}
\usepackage{mathtools}
\usepackage{bbm}
\usepackage{esint}

\usepackage[margin=2.6cm]{geometry}
\usepackage{setspace}
\usepackage{indentfirst}
\usepackage{graphicx}
\usepackage{graphics}
\usepackage{lscape}
\usepackage{pgf,tikz}
\usepackage{tikz-cd}
\usepackage{color}
\usepackage{pict2e}
\usepackage{epic}
\usepackage{epstopdf}
\usepackage{titlesec, titlefoot}
	\titleformat{\section}[block]{\Large\bfseries\filcenter}{\thesection}{1em}{}
\usepackage{commath}
\usepackage{float}
\usepackage{caption}
\usepackage{etoolbox}
\usepackage[affil-it]{authblk}
\usepackage{combelow}

\usepackage[  colorlinks=true,
  allcolors=black,
  urlcolor=blue,
  bookmarksdepth=3
  ]{hyperref}
\hypersetup{bookmarksopen=true} 
\usepackage{hypcap}

\graphicspath{{./Pictures/}}
\allowdisplaybreaks

\expandafter\def\expandafter\normalsize\expandafter{%
    \normalsize
    \setlength\abovedisplayskip{6pt}
    \setlength\belowdisplayskip{6pt}
    \setlength\abovedisplayshortskip{6pt}
    \setlength\belowdisplayshortskip{6pt}
}

\theoremstyle{plain}

\renewcommand*\thesection{\arabic{section}}
\numberwithin{equation}{section}

\newtheorem{theorem}{Theorem}[section]
\newtheorem{lemma}[theorem]{Lemma}
\newtheorem*{lemma*}{Lemma}

\theoremstyle{definition}

\newtheorem{remark}[theorem]{Remark}

\expandafter\let\expandafter\oldproof\csname\string\proof\endcsname
\let\oldendproof\endproof
\renewenvironment{proof}[1][\proofname]{%
  \oldproof[\upshape \bfseries #1]%
}{\oldendproof}
\makeatletter
\def\@makechapterhead#1{%
  \vspace*{50\p@}%
  {\parindent \z@ \raggedright \normalfont
    \interlinepenalty\@M
    \Huge\bfseries  \thechapter.\quad #1\par\nobreak
    \vskip 40\p@
  }}
\makeatother

\def \a{\alpha}
\def \R {\mathbb{R}}
\newcommand{\B}{\mathbb B}

\def \D{\textup{D}}

\def \e{\varepsilon}
\def \d{\textup{d}}

\def \p{\partial}
\def \mc{\mathcal}
\def \mb{\mathbb}

\def \tp{\textup}

\def \loc{\textup{loc}}

\newcommand{\Sn}{\mathbb S^{n-1}}

\DeclareMathOperator{\ddiv}{div}

\DeclareMathOperator{\Id}{Id}

\DeclareMathOperator{\sgn}{sgn}

\begin{document}

	\title{\textbf{On the sharp H\"older exponent\\ in the De Giorgi--Nash--Moser theory}}
		
	\author{{\Large Andr\'e Guerra}}
		
	\affil{\small Department of Pure Mathematics and Mathematical Statistics,  University of Cambridge,\protect\\  Wilberforce Rd, Cambridge CB3 0WB, UK
	\protect\\
	{\tt{adblg2@cam.ac.uk}}  }
			
	\date{}
	
	\maketitle

%	\vspace{-.5cm}	
	
	\begin{abstract}
	We consider solutions of uniformly elliptic equations with measurable coefficients. We assume that the lowest eigenvalue of the coefficient matrix is at least $K^{-1}$ and the largest eigenvalue is at most $K$. In three and higher dimensions we construct $\alpha$-H\"older continuous solutions with $\alpha = \exp(- c_n K)$. This disproves a long-standing conjecture by showing that, except for the two-dimensional case, the H\"older exponent obtained from the Bombieri--Giusti Harnack inequality has the optimal dependence on the ellipticity constant $K$. Using the same construction as a starting point, we disprove a conjecture of De Giorgi about continuity of solutions to non-uniformly elliptic equations.
	\end{abstract}

\section{Introduction}

\subsection{Uniformly elliptic equations}
In this note we consider weak solutions $u\in W^{1,2}(\mb B^n)$ of uniformly elliptic equations
\begin{equation}
\label{eq:PDE}
\ddiv(A\,\D u)=0 \quad \text{in } \mb B^n.
\end{equation}
Here and throughout, $A\colon \mb B^n\to \textup{Sym}_n$ is a measurable symmetric matrix field satisfying 
\begin{equation}
\label{eq:ell}
\frac 1 K |\xi|^2 \leq A(x)\xi\cdot \xi \leq K |\xi|^2 \quad \text{for a.e.\ } x\in \mb B^n, \xi \in \R^n.
\end{equation}
In their seminal works, De Giorgi \cite{DeGiorgi1957} and Nash \cite{Nash1958} independently proved that weak solutions of \eqref{eq:PDE} are H\"older continuous,
\begin{equation}
\label{eq:Holder}
u\in C^{0,\alpha}(\tfrac 1 2 \mb B^n) \quad \text{for some } \alpha=\alpha(n,K)\in (0,1),
\end{equation}
leading to the solution of Hilbert's 19th problem on the regularity of minimizers of scalar uniformly convex variational problems.
A few years later, Moser \cite{Moser1960,Moser1961} gave another proof of \eqref{eq:Holder}, and in fact he obtained the stronger result that non-negative solutions of \eqref{eq:PDE} satisfy a Harnack inequality:
\begin{equation}
\label{eq:Harnack}
\sup_{\frac 1 2 \mb B^n} u \leq C_H(n,K) \inf_{\frac 1 2 \mb B^n} u.
\end{equation}
These are by now classical results, which the reader can find in a variety of textbooks on elliptic PDEs, and which have recently been formalized \cite{Armstrong2026}. 

The truly remarkable aspect of the De Giorgi--Nash--Moser theory is that the coefficients of \eqref{eq:PDE} are just \textit{measurable}. This is unlike Schauder theory, where one considers equations with \textit{continuous} coefficients, which can be seen as perturbations of equations with \textit{constant} coefficients. Instead, \eqref{eq:Holder}--\eqref{eq:Harnack} are non-perturbative results, as they follow from the ellipticity bounds \eqref{eq:ell} alone. Hence it is natural to ask how these results depend on the ellipticity constant $K$. 

The case of the Harnack inequality \eqref{eq:Harnack} is well understood, after the beautiful work of Bombieri--Giusti \cite{Bombieri1972}. Indeed, they showed that one can take
\begin{equation}
\label{eq:BG}
C_H(n,K) = \exp(c_n K).
\end{equation}
The dependence on $K$ is sharp already in dimension $n=2$, since
$$\frac 1 K \p_{11} u + K \p_{22} u=0, \qquad u(x_1,x_2)=e^{K x_1} \cos(x_2).$$
When $n=2$, one can even express the \textit{optimal constant} $C_H(2,K)$ in \eqref{eq:Harnack} in terms of elliptic integrals. Indeed, in two dimensions a solution $u$ of \eqref{eq:PDE} can be factorized as 
\begin{equation}
\label{eq:factorization}
u=h\circ f, \qquad h \text{ is harmonic, $\quad f$ is $K$-quasiconformal},
\end{equation}
see \cite[Theorem 16.2.1]{Astala2009}. The sharp form of the Harnack inequality then follows from the sharp Harnack inequality for harmonic functions and the sharp Schwarz Lemma for quasiconformal maps, see \cite[Theorem 16.2.2]{Astala2009} and \cite[Theorem 16.39]{Hariri2020}.

Instead, less is known about the H\"older exponent in \eqref{eq:Holder}. The two-dimensional case is classical and precedes the De Giorgi--Nash--Moser theory: it was shown by Morrey \cite{Morrey1938} that $K$-quasiconformal maps are $1/K$-H\"older continuous, hence by \eqref{eq:factorization} we have 
\begin{equation}
\label{eq:2d}
\frac 1 K \leq \alpha(2,K).
\end{equation}
A direct proof of this result was also given in an elegant paper by Piccinini--Spagnolo \cite{Piccinini1972}. Simple examples show that \eqref{eq:2d} is \textit{optimal}, and the natural extension of these examples to higher dimensions shows more generally that one can only have, for arbitrary solutions,
\begin{equation}
\label{eq:nd}
\alpha(n,K) \leq \frac 1 2 \left(\sqrt{(n-2)^2+ \frac{4(n-1)}{K^2}}-(n-2)\right)\leq \begin{cases}  \frac 1 K &\text{if } n=2, \\  \frac{2}{K^2} &\text{if } n\geq 3,\end{cases}
\end{equation}
%the last inequality holding for $n\geq 3$,
see \cite[\S 4]{Piccinini1972}. For operators with a particular structure, this exponent is optimal \cite{Greco2002}. Interestingly, the same examples also display the sharp dependence on $K$ in the $L^\infty$ estimate for solutions of \eqref{eq:PDE}, see \cite[Proposition 1.5]{Bella2024}, at least when $n\geq 4$.

We emphasize that the examples of Piccinini--Spagnolo are much more regular than what follows from the sharp Harnack inequality \eqref{eq:BG}, which gives only the bound
\begin{equation}
\label{eq:BGholder}
\exp(-c_n K )\leq \alpha(n,K).
\end{equation}
A conjecture in the field, often attributed to De Giorgi, asserts that this lower bound is highly suboptimal and that, instead, the optimal H\"older exponent has algebraic dependence on the ellipticity ratio, similarly to the examples of Piccinini--Spagnolo, see \cite[p.\ 70]{Armstrong2019} and \cite{Armstrong2022,Mosconi2018}. 
In this paper we give an example showing that, to the contrary, \eqref{eq:BGholder} is sharp as soon as $n\geq 3$. Precisely, we have:

\begin{theorem}\label{thm:main}
Let $n\geq 3$. For each $\gamma\geq 1$ there is a measurable symmetric matrix field $A_\gamma\in L^\infty(\mb B^n, \textup{Sym}_n)$ satisfying the ellipticity conditions
\begin{equation}
\label{eq:ellmain}
\frac 1 {K_\gamma} |\xi|^2 \leq A_\gamma(x)\xi\cdot \xi\leq  K_\gamma |\xi|^2\quad \text{for a.e.\ } x\in \mb B^n, \xi \in \R^n,
\end{equation}
and a positively $\alpha_\gamma$-homogeneous function $u_\gamma\in W^{1,2}(\mb B^n)$ which is a weak solution of
\begin{equation}
\label{eq:PDEmain}
\ddiv(A_\gamma \D u_\gamma)=0 \quad \text{in } \mb B^n,
\end{equation}
and such that, for some constants $C_n,c_n>0$,
$$\alpha_\gamma \leq \exp(-c_n K_\gamma), \qquad K_\gamma = C_n \gamma.$$
\end{theorem}

The example is, in a sense, as simple as possible: the solution and the operator are positively homogeneous, hence the PDE becomes an eigenvalue problem on the sphere. We choose the operator in such a way that this is an eigenvalue problem for a weighted Laplacian on $\mb S^{n-1}$, with weight $e^{-t V}$ for $t\approx \gamma$ and $V$ a double-well potential. The condition $n\geq 3$ is needed only to solve the eigenvalue problem, and the solution of the PDE is the $\a$-homogeneous extension of the corresponding first eigenfunction. The two wells of $V$ guarantee that the eigenvalues of the weighted Laplacian are exponentially small, and hence so is $\alpha_\gamma$. This type of principle is well-known, see e.g.\ \cite{Freidlin1998,Holley1989} for much more sophisticated results, which may be useful if one wants to determine the \textit{optimal constant} $c_n$ in \eqref{eq:BGholder} as $K\to \infty$. 

\medskip
We conclude by mentioning two problems that, although related, are quite different from the one considered here. The first is to consider non-divergence form equations: even in two dimensions, the optimal H\"older exponent is not known \cite{Baernstein2005}. The second concerns higher integrability: it was shown by Meyers that solutions of \eqref{eq:PDE}--\eqref{eq:ell} are in $W^{1,p}_\loc(\mb B^n)$ for some $p=p(n,K)>2$ \cite{Meyers1963}, and in fact Iwaniec--Sbordone proved in \cite{Iwaniec2001a} that $p(n,K)\geq 2 + \frac{c}{K-1}$ for a \textit{universal} constant $c\geq \frac 2 7$. Up to determining the optimal value of $c$, this result is optimal; the optimal value is conjectured to be $c=2$, but this is only known when $n=2$, due to Astala's Area Distortion Theorem \cite[\S 13]{Astala2009}. 

\subsection{Non-uniformly elliptic equations}

We now turn to non-uniformly elliptic equations. This subsection was added in a revised version of the paper, after one of the referees brought several important references to the author's attention.

To formulate the problem and describe the state of the art, we consider the following generalization of \eqref{eq:ell}. Suppose there are measurable functions $0<\lambda\leq 1\leq \Lambda<\infty$ such that
\begin{equation}
\label{eq:nonunifell}
\lambda(x) |\xi|^2 \leq A(x)\xi\cdot \xi \leq \Lambda(x) |\xi|^2  \qquad \text{for a.e.\ } x\in \mb B^n, \xi \in \R^n.
\end{equation}
When both $\lambda^{-1}$ and $\Lambda$ are essentially bounded by $K$, condition \eqref{eq:nonunifell} is equivalent to \eqref{eq:ell}, but here we just assume that these functions are integrable. Let $H^1(\B^n,A)$ denote the completion of
$C^1(\B^n)$ under the norm
\begin{equation}\label{eq:energy-norm}
  \|v\|_{H^1(\B^n,A)}^2
  \equiv \int_{\B^n}(A\nabla v\cdot\nabla v+\Lambda v^2)\,\d x,
\end{equation}
and let
$H^1_0(\B^n,A)$ be the corresponding closure of
$C_c^1(\B^n)$. A function $u\in H^1(\B^n,A)$ is said to be a \textit{weak
solution} of \eqref{eq:PDE} if
\[
  \int_{\B^n}A\nabla u\cdot\nabla\varphi\,\d x=0
  \qquad\text{for every }\varphi\in H^1_0(\B^n,A).
\]
The notation $H^1(\mb B^n,A)$, rather than $W^{1,2}(\mb B^n,A)$, is deliberate. In this setting there is, in general, no analogue of the Meyers--Serrin ``$H=W$'' theorem: smooth functions need not be dense in the space of functions with finite $\|\cdot\|_{H^1(\mb B^n,A)}$-norm, see \cite[\S 2]{Franchi1998}.

Several extensions of the De Giorgi--Nash--Moser theory to the non-uniformly elliptic setting are available; see, for instance, \cite{Armstrong2026a,Bella2021,Bella2024,Fabes1982,Trudinger1971}. Concerning the Harnack inequality, the strongest result  in this direction is established in \cite{Bella2021}: if $u$ is a non-negative weak solution of \eqref{eq:PDE} and
\begin{equation}
\label{eq:exprange}
\frac 1 \lambda \in L^q(\mb B^n), \qquad \Lambda\in L^p(\mb B^n), \qquad \frac 1 p + \frac 1 q< \frac{2}{n-1},
\end{equation}
for some $p,q\in (1,\infty]$, then
\begin{equation}
\label{eq:harnackdeg}
\sup_{B_r} u\leq C(n,p,q,\mathcal K(B_r)) \inf_{B_r} u, \qquad \mathcal K(E) \equiv \bigg( \fint_E \Lambda^p\bigg)^\frac 1 p \bigg(\fint_E \lambda^{-q}\bigg)^\frac 1 q
\end{equation}
for every $r\in (0,\frac 1 2)$. The range of exponents in \eqref{eq:exprange} is optimal, at least when $n\geq 4$ \cite{Franchi1998}. If one makes the additional assumption that $\lambda,\Lambda$ are $A_2$-weights, then one can obtain a dependence of the constant in \eqref{eq:harnackdeg} similar to that in \eqref{eq:BG}, see \cite[Theorem B and (1.4)]{Chanillo1986}.

Because the constant in \eqref{eq:harnackdeg} depends on $\mathcal K(B_r)$, estimate \eqref{eq:harnackdeg} does not in general imply H\"older continuity at arbitrarily small scales; it yields such regularity only on large scales, as explained in \cite[Corollary 4.2]{Bella2021}. For $n\geq 3$, to be able to obtain continuity by these methods essentially requires double-exponential integrability, see e.g.\ \cite[Theorem 5.2]{Trudinger1971} or \cite[Theorem 5.3]{Chanillo1986}. In fact, it follows from \cite{Armstrong2026a} that the following more precise condition
\begin{equation}
\label{eq:doubleexp}
\int_\Omega \exp \exp(c \lambda^{-\frac p2})  \, \d x + \int_\Omega \exp \exp(c \Lambda^{\frac q2})  \, \d x <\infty,
\end{equation}
where $p,q\in(1,\infty)$, $\frac 1 p + \frac 1 q =1$, and $c=c(n,p,q)>0$ is sufficiently large, suffices for continuity of solutions. At an endpoint, the condition on the uniformly controlled side is omitted; in particular, when $\lambda=1$ one may take $p=\infty$ and $q=1$. See also Remark \ref{rem:coarse} below.

In the notes \cite{DeGiorgi1995}, which initiated the circle of problems considered in this paper, De Giorgi conjectured a much stronger result, namely that, when $n\geq 3$, the condition
\begin{equation}
\label{eq:expintegrability}
\lambda = 1, \qquad  \int_\Omega \exp(\Lambda) \, \d x <\infty 
\end{equation}
should imply continuity of solutions to \eqref{eq:PDE}. He also conjectured that this condition cannot be weakened and proposed \textit{isotropic} coefficients---that is, coefficients of the form $A(x)=a(x)\Id$ for some scalar function $a\colon \Omega\to (0,\infty)$---which just fail \eqref{eq:expintegrability} and admit discontinuous solutions. This second conjecture was proved in \cite{Zhong2008}, which also contains the statements of De Giorgi's conjectures. See also \cite{Astala2010a,Franchi1998,Onninen2007} for related results.

After learning of these references through a referee report, the author felt that De Giorgi's continuity conjecture under \eqref{eq:expintegrability} was overly optimistic. The isotropic setting, proposed by De Giorgi as a test of the sharpness of \eqref{eq:expintegrability}, is typically much more rigid; see, for example, \cite[\S 3]{Piccinini1972}. By contrast, the example in Theorem \ref{thm:main} is far from isotropic. With the help of ChatGPT 5.6-Pro, the author could then modify the construction of Theorem \ref{thm:main} in order to disprove De Giorgi's conjecture. In fact, the following holds:

%This suggested a simple experiment: could an AI model, with \textit{minimal human assistance}, adapt that construction to disprove De Giorgi's continuity conjecture under \eqref{eq:expintegrability}? Traditionally, such a problem would have made an excellent PhD project: it requires a non-trivial modification of a recent construction and leads to a counterexample to a significant conjecture.
%The author therefore obtained a subscription to ChatGPT Pro, and gave the model only extremely limited guidance, similar to what a PhD supervisor might say to a student:
%\begin{quote}
%\textit{Give a counterexample to Conjecture 3 here:  \cite{Zhong2008}. To do so, adapt the example in [present manuscript]. Don't use the internet other than to look up these references.}
%\end{quote}
%The complete conversation is available at \cite{GTPchat2026}. After 25 minutes, the LLM produced an affirmative answer. Although its initial response was difficult to read, several requests for revision yielded an argument that the author was able to verify as a correct proof of the following result:

\begin{theorem}\label{thm:GPT}
Let $n\geq 3$. There exist a measurable symmetric matrix field
$A\colon \mb B^n\to \tp{Sym}_n$ with
\begin{equation}
\label{eq:elldeg}
|\xi|^2\leq A(x)\xi\cdot\xi
\leq \Lambda(x)|\xi|^2
\qquad\text{for a.e.\ }x\in\mb B^n,\ \xi\in\R^n,
\end{equation}
where, with $r=|x|$, we have
\begin{equation}
\label{eq:Lambda}
\Lambda(r)=C_n \log\Bigl(e+ \log \frac 1 r\Bigr)^2,
\end{equation}
and a weak solution
$u\in H^1(\mb B^n,A)$ of \eqref{eq:PDE} with no continuous representative at the origin.
%\begin{equation}
%\label{eq:expint}
%\int_{\mb B^n}\exp \Lambda(x)\,\d x<\infty,
%\end{equation}
\end{theorem}

Note that, by \eqref{eq:Lambda}, we in particular have 
$\int_{\mb B^n}\exp(C\Lambda^p)\,\d x<\infty$ for all $p,C<\infty$, which is much stronger than \eqref{eq:expintegrability}. In fact, \eqref{eq:Lambda} also shows that
$$
\int_{\mb B^n} \exp (\exp(c \sqrt{\Lambda})) \, \d x <\infty, \qquad c\ll 1,
$$
and thus Theorem \ref{thm:GPT} implies that, up to finding the best value of $c$, \eqref{eq:doubleexp} is optimal: already in the case where $p=\infty, q=1$, if $c$ is too small then solutions are not necessarily continuous.

%To be precise, the original proof of ChatGPT concerned the symmetric case
%where $1/\lambda=\Lambda$ in \eqref{eq:nonunifell}, which is the one
%considered in \cite[Conjecture 3]{Zhong2008}, but 
%We remark that a trivial adaptation of
%the proof of Theorem \ref{thm:GPT} also leads to the stronger statement in Theorem \ref{thm:GPT}, with $\lambda=1\leq \Lambda$,
%which disproves \cite[Conjecture 1]{Zhong2008}. Similarly, one can also disprove \cite[Conjecture 2]{Zhong2008} for the case $\lambda\leq 1=\Lambda$.

%A proof of Theorem \ref{thm:GPT} is given in the appendix. It follows, in essence, the argument generated by ChatGPT-5.6 Pro, but has been fully rewritten and verified by the author. In the author's opinion, this is a fairly non-trivial modification of the construction in Theorem \ref{thm:main}, but we leave that judgement to the reader.

\subsection*{Acknowledgements}
AG thanks the Royal Society for support through a Newton International Fellowship. He thanks one of the referees for bringing several relevant references to his attention and for many insightful comments, S.\ Armstrong and T.\ Kuusi for checking the correctness of the first version of the paper and providing useful feedback, and S.\ Armstrong and J.\ de Dios Pont for very helpful discussions about AI as well as suggestions for how to rewrite the second version of this manuscript.

\subsection*{AI disclosure}
In the first version of this paper, ChatGPT-5.5 Plus had been used to proofread the text and look for typos and mistakes in its contents. Before writing the current version of the paper, the author obtained a subscription to ChatGPT-5.6 Pro, and used it to generate, essentially autonomously, a proof of Theorem \ref{thm:GPT}. The original proof was generated quite quickly (in 25 minutes) but it was extremely hard to read, and it took the author several days to properly understand and verify it. The proof presented here was written entirely by the author, who takes responsibility for its correctness.

%This suggested a simple experiment: could an AI model, with \textit{minimal human assistance}, adapt that construction to disprove De Giorgi's continuity conjecture under \eqref{eq:expintegrability}? Traditionally, such a problem would have made an excellent PhD project: it requires a non-trivial modification of a recent construction and leads to a counterexample to a significant conjecture.
%The author therefore obtained a subscription to ChatGPT Pro, and gave the model only extremely limited guidance, similar to what a PhD supervisor might say to a student:
%\begin{quote}
%\textit{Give a counterexample to Conjecture 3 here:  \cite{Zhong2008}. To do so, adapt the example in [present manuscript]. Don't use the internet other than to look up these references.}
%\end{quote}
%The complete conversation is available at \cite{GTPchat2026}. After 25 minutes, the LLM produced an affirmative answer. Although its initial response was difficult to read, several requests for revision yielded an argument that the author was able to verify as a correct proof of the following result:

%\clearpage
\section{Proof of  Theorem \ref{thm:main}}\label{sec:proof}

We begin by remarking that, although we formulated our results in the symmetric case for simplicity, the De Giorgi--Nash--Moser theory also holds for equations with non-symmetric coefficients. In fact, in proving Theorem \ref{thm:main} it will be useful to first construct a non-symmetric example. The next simple lemma shows that one can always reconstruct a symmetric example a posteriori:

\begin{lemma}\label{lemma:sym}
Let $B\in L^\infty(\mb B^n, \R^{n\times n})$ satisfy, for some $K\geq 1$,
\begin{equation}
\label{eq:ellB}
B(x)\xi\cdot \xi\geq \frac 1 K |\xi|^2 \quad \text{ and } \quad  B^{-1}(x)\xi\cdot \xi \geq \frac 1 K |\xi|^2\quad \text{for a.e.\ } x\in \mb B^n, \xi \in \R^n
\end{equation}
and let $u\in W^{1,2}(\mb B^n)$ be a weak solution of $\ddiv(B\,\D u)=0$. Then there is a symmetric matrix field $A \in L^\infty(\mb B^n,\textup{Sym}_n)$, satisfying \eqref{eq:ell} with $2K$ in place of $K$, and such that $\ddiv(A\, \D u)=0$ weakly.
\end{lemma}

\begin{proof}
At a point $x$ where \eqref{eq:ellB} holds and $\D u(x)\neq 0$, we set
$$A(x) \equiv  a(x) \Id + \left[e(x) \otimes z(x) + z(x) \otimes e(x)\right]+ \frac{2}{a(x)} z(x)\otimes z(x),$$
where we set
$$
e \equiv  \frac{\D u}{|\D u|},  \qquad a\equiv  Be\cdot e,\qquad z\equiv Be-a \,e. $$
We set $A(x) \equiv  K^{-1} \Id$ at points where $\D u(x)=0$, and also on the remaining null set.
Clearly $A$ is symmetric. Moreover $z\cdot e=0$ and, since $\D u = |\D u| e$, we have
$$A\, \D u=|\D u|(ae+z)=B\,\D u,$$
for a.e.\ $x$, and so $\ddiv(A\,\D u)=0$ weakly.

It remains to check the ellipticity condition. Fix a point $x$ where \eqref{eq:ellB} holds and $\D u(x)\neq 0$. From \eqref{eq:ellB} applied with $\xi=e$ and $\xi=B e$ we get
$$a\geq \frac1K, \qquad \frac{|Be|^2}{a}\leq K.$$
If $z=0$, then $A=a\Id$ and the desired bound follows from $K^{-1}\leq a\leq K$. Otherwise, in the plane spanned by $e$ and $\tfrac{z}{|z|}$, $A$ can be represented as 
$$\begin{bmatrix}
a & |z| \\ |z| & a + 2\frac{|z|^2}{a}
\end{bmatrix}
= a \begin{bmatrix}
1 & t \\ t & 1 + 2t^2
\end{bmatrix},
\qquad t\equiv \frac{|z|}{a},
$$
which has eigenvalues
$$a\left(1+t^2\pm t\sqrt{1+t^2}\right).$$
On the orthogonal complement of this 2-dimensional space, $A$ has a single eigenvalue $a$. Since
$$\frac 1 2 \leq 1+t^2-t\sqrt{1+t^2},\qquad
1+t^2+t\sqrt{1+t^2}\leq 2(1+t^2),$$
for any $\xi\in \R^n$ we obtain
$$\frac1{2K}|\xi|^2\leq A(x)\xi\cdot \xi\leq 2K|\xi|^2.$$
The same bound is immediate at points where $A=K^{-1}\Id$.
\end{proof}

We now proceed with the proof of the main result. We will not keep track of the dependence of constants on $n$, and so constants with the same notation may change from line to line.

\begin{proof}[Proof of Theorem \ref{thm:main}]
Write $x=r \omega$ with $\omega \in \mb S^{n-1}$, and let us decompose vectors into their normal and tangential parts:
$$\xi = \xi_r e_r + \xi_T, \qquad \xi_T \in T_\omega \mb S^{n-1}.$$
Let us fix $M$, to be chosen large depending only on $n$ later on, and define
$$V(\omega) \equiv  1-\omega_n^2, \qquad 
        Y_\gamma(\omega)\equiv -\gamma\, \D_T V(\omega),
        \qquad
        b_\gamma\equiv M\gamma^2,
$$
where $\D_T$ denotes the tangential derivative on $\mb S^{n-1}$. Define $B_\gamma$ in the frame
$(e_r,T_\omega\mb S^{n-1})$ by
\[
        B_\gamma(\omega)\equiv 
        \begin{bmatrix}
        b_\gamma & Y_\gamma(\omega)^T\\
        0 & \Id_{T_\omega\mb S^{n-1}}
        \end{bmatrix}
\]
so that
\begin{equation}
\label{eq:defB}
        B_\gamma(\xi_r,\xi_T)
        =
        \bigl(b_\gamma\xi_r+Y_\gamma\cdot\xi_T,\,\xi_T\bigr).
\end{equation}
Equivalently, we can define $B_\gamma$ in coordinate-free notation as
$$        B_\gamma(\omega)=  \omega\otimes (b_\gamma \omega+Y_\gamma(\omega)) + (\Id-\,\omega\otimes \omega).$$
Note that $B_\gamma$ is a positively 0-homogeneous matrix field which is smooth away from the origin.
We now divide the rest of the proof into several steps.

\medskip
\textbf{Step 1: calculation of the ellipticity constant.}
Since $|Y_\gamma|\le C_n\gamma$, we have
$$|\xi_r Y_\gamma\cdot \xi_T|\leq \frac 1 2 b_\gamma \xi^2_r + \frac{|Y_\gamma|^2}{2 b_\gamma}|\xi_T|^2
\leq \frac 1 2 b_\gamma \xi^2_r + \frac{C_n}{M}|\xi_T|^2$$
and therefore, choosing $M$ sufficiently large so that $C_n/M\leq \frac 1 2$, we have
\begin{align*}
        B_\gamma\xi\cdot \xi
        &=b_\gamma\xi_r^2+|\xi_T|^2+
          \xi_rY_\gamma\cdot\xi_T  \ge \frac 1 2 \bigl(b_\gamma\xi_r^2+|\xi_T|^2\bigr)
        \ge \frac 1 2 |\xi|^2
\end{align*}
since $\gamma\geq 1$.
Moreover
\[
        B_\gamma^{-1}=
        \begin{bmatrix}
        b_\gamma^{-1} & -b_\gamma^{-1}Y_\gamma^T\\
        0 & \Id_{T_\omega \mb S^{n-1}}
        \end{bmatrix},
\]
and, just as above, we have
$$|b_\gamma^{-1}\xi_r Y_\gamma\cdot \xi_T|
\leq \frac 1 2 b_\gamma^{-1}\xi_r^2 + \frac{|Y_\gamma|^2}{2 b_\gamma}|\xi_T|^2
\leq \frac 1 2 b_\gamma^{-1}\xi_r^2 + \frac{C_n}{M}|\xi_T|^2.$$
Thus, again as above, and by the definition of $b_\gamma$, we have
\[
        B_\gamma^{-1}\xi\cdot \xi
        \ge
        \frac 1 2 \bigl(b_\gamma^{-1}\xi_r^2+|\xi_T|^2\bigr)
        \ge \frac {1}{c_n\gamma^2}|\xi|^2 .
\]
Thus, for $\widehat B_\gamma\equiv \gamma^{-1}B_\gamma$,
\[
        \widehat B_\gamma\xi\cdot\xi\ge \frac{1}{2\gamma}|\xi|^2,
        \qquad
        \widehat B_\gamma^{-1}\xi\cdot\xi\ge \frac{1}{c_n\gamma}|\xi|^2.
\]
That is, $\widehat B_\gamma$ satisfies \eqref{eq:ellB} with $K=c_n\gamma$.

\medskip
\textbf{Step 2: reduction to an eigenvalue problem.}
We now look for positively homogeneous solutions
\[
        u(r,\omega)=r^\alpha\phi(\omega)
\]
of $\ddiv(B_\gamma \D u)=0$, where $\phi$ is smooth.
We compute
\begin{equation}
\label{eq:gradu}
        \D u
        =r^{\alpha-1}
        \bigl[\alpha\phi\,e_r+\D_T\phi\bigr],
        \qquad 
                B_\gamma\,\D u
        =r^{\alpha-1}
        \bigl[\left(b_\gamma\alpha\phi+Y_\gamma\cdot\D_T\phi\right) e_r + \D_T\phi\bigr].
\end{equation}
Recalling that, if $X=X_r e_r+X_T$, with $X_T(r,\omega)\in T_\omega \mb S^{n-1}$, is a vector field, then
$$
        \ddiv X
        =r^{-(n-1)}\partial_r(r^{n-1}X_r)
        +r^{-1}\ddiv_T X_T,$$
where $\ddiv_T$ is the tangential divergence on $\mb S^{n-1}$,
we obtain
\[
        \ddiv(B_\gamma\D u)
        =r^{\alpha-2}
        \left[
        \Delta_T \phi
        +(\alpha+n-2)Y_\gamma\cdot\D_T\phi
        +b_\gamma\alpha(\alpha+n-2)\phi
        \right],
\]
where $\Delta_T=\ddiv_T \D_T$ is the usual Laplace--Beltrami operator on $\mb S^{n-1}$.
Because $Y_\gamma=-\gamma\,\D_T V$, the equation $\ddiv(B_\gamma \D u)=0$ is equivalent, away from the origin, to
\begin{equation}
\label{eq:PDEreduced}
        -\left(
        \Delta_T-(\alpha+n-2)\gamma\,\D_T V\cdot\D_T
        \right)\phi
        =b_\gamma\alpha(\alpha+n-2)\phi .
\end{equation}
For $t\geq 1$, we define the self-adjoint operator
\[
        \mc L_t\phi
        \equiv -\Delta_T\phi+t\,\D_T V\cdot\D_T\phi
        =-e^{tV}\ddiv_T(e^{-tV}\D_T\phi)
\]
 in $L^2(\Sn,e^{-tV}\d\sigma)$,
where $\sigma$ is the surface measure on $\Sn$. The operator $\mc L_t$ is sometimes known as the \textit{weighted Laplacian} or \textit{Bakry--\'Emery Laplacian}. 
%for some of its basic spectral properties.

\medskip
\textbf{Step 3: continuity of the first eigenvalue.}
Let $\mu_1(t)>0$ be the first positive eigenvalue of $\mc L_t$, characterized through the usual Rayleigh quotient:
\begin{equation}
\label{eq:rayleigh}
        \mu_1(t)= 
        \inf\left\{
        \frac{\int_{\Sn}|\D_T\psi|^2 e^{-tV}\d\sigma}
             {\int_{\Sn}\psi^2 e^{-tV}\d\sigma}
        : \psi\in W^{1,2}(\Sn),\psi\neq 0, \int_{\Sn}\psi e^{-tV}\d\sigma=0
        \right\}.
\end{equation}
Note that, for fixed $t$, the weight $w_t\equiv e^{-tV}$ is comparable to 1, hence the weighted Sobolev space is equivalent to the usual one. We claim that $t\mapsto \mu_1(t)$ is continuous on $[1,\infty)$. Indeed, since averages minimize the $L^2$-distance to constants, the Rayleigh quotient can be written as
\[
        \mu_1(t)=
        \inf_{\psi\neq \mathrm{const}}
        \frac{\int_{\Sn}|\D_T\psi|^2 w_t\,\d\sigma}
             {\min_{c\in \R}\int_{\Sn}|\psi-c|^2 w_t\,\d\sigma}.
\]
Fix $s,t\geq 1$ and set $\delta=|s-t|$. Since $0\leq V\leq 1$, we have
\[
        e^{-\delta}w_t\leq w_s\leq e^\delta w_t.
\]
Therefore, for every non-constant $\psi$,
\[
        e^{-\delta}\int_{\Sn}|\D_T\psi|^2 w_t\,\d\sigma
        \leq
        \int_{\Sn}|\D_T\psi|^2 w_s\,\d\sigma
        \leq
        e^\delta\int_{\Sn}|\D_T\psi|^2 w_t\,\d\sigma
\]
and also
\[
        e^{-\delta}\min_{c\in \R}\int_{\Sn}|\psi-c|^2 w_t\,\d\sigma
        \leq
        \min_{c\in \R}\int_{\Sn}|\psi-c|^2 w_s\,\d\sigma
        \leq
        e^\delta\min_{c\in \R}\int_{\Sn}|\psi-c|^2 w_t\,\d\sigma .
\]
Taking the infimum over $\psi$ gives
\[
        e^{-2\delta}\mu_1(t)\leq \mu_1(s)\leq e^{2\delta}\mu_1(t),
\]
which proves the claim.

\medskip
\textbf{Step 4: bounds on the first eigenvalue.} We claim that, for $t\geq 1$ and some $c>0$,
\begin{equation}
\label{eq:boundmu}
        \mu_1(t)\le C_n e^{-c t}.
\end{equation}
To see this, choose $\eta\in C^\infty([-1,1])$ odd, with
\[
        \eta(s)=1\quad\text{for }s\ge\frac23,
        \qquad
        \eta(s)=-1\quad\text{for }s\le-\frac23,
        \qquad
        |\eta'|\le 4.
\]
Let $\psi(\omega)=\eta(\omega_n)$. Since $V(\omega)=1-\omega_n^2$ is even in $\omega_n$ and $\psi$ is odd,
\[
        \int_{\Sn}\psi e^{-tV}\d\sigma=0.
\]
The integrand in the numerator of \eqref{eq:rayleigh} is supported where $|\omega_n|\le 2/3$, and there we have $V\ge 5/9$. Thus
\begin{equation}
\label{eq:numerator}
        \int_{\Sn}|\D_T\psi|^2e^{-tV}\d\sigma
        \le C_n e^{-5t/9}.
\end{equation}
For the denominator in \eqref{eq:rayleigh}, we restrict to the cap where
\[
        \omega_n\ge \max\{\tfrac 2 3,\sqrt{1-t^{-1}}\}.
\]
The second condition is equivalent to $V\leq \frac 1 t$ and guarantees that the weight is at least $e^{-1}$ in the cap; also
$\psi=1$ in the cap. As the cap has surface area comparable to $t^{-(n-1)/2}$, we have
\begin{equation}
\label{eq:denominator}
        \int_{\Sn}\psi^2e^{-tV}\d\sigma
        \ge c_n t^{-(n-1)/2}.
\end{equation}
Combining \eqref{eq:numerator}--\eqref{eq:denominator}, and absorbing the polynomial bound into the exponential, by choosing any $c<\frac 5 9$ we obtain  \eqref{eq:boundmu}.

\medskip
\textbf{Step 5: solving the eigenvalue problem.}
This is where we use the condition $n\geq 3$.
For $\alpha\geq 0$, let
\[
        F_\gamma(\alpha)\equiv 
        b_\gamma\alpha(\alpha+n-2)-\mu_1((\alpha+n-2)\gamma).
\]
Since $n\geq 3$ and $\gamma\geq 1$, the argument of $\mu_1$ is at least $1$ for every $\alpha\geq 0$,  and we also have
\[
        F_\gamma(0)=-\mu_1((n-2)\gamma)<0.
\]
Also, by \eqref{eq:boundmu}, we have $\mu_1((\alpha+n-2)\gamma)\le C_n$ for all $\alpha\geq 0$.
Hence, choosing $N=N_n$ large,
\[
        F_\gamma(Nb_\gamma^{-1})
        \ge N(n-2)-C_n>0,
\]
again since $n\geq 3$.
Thus, by continuity of $F_\gamma$, which follows from Step 3, the Intermediate Value Theorem guarantees the existence of $\alpha_\gamma\in (0,Nb_\gamma^{-1})$ with $F_\gamma(\alpha_\gamma)=0$, i.e.\
\begin{equation}
\label{eq:eigenvalueidentity}
        b_\gamma\alpha_\gamma(\alpha_\gamma+n-2)=\mu_1((\alpha_\gamma+n-2)\gamma).
\end{equation}
By \eqref{eq:boundmu}, for $M$ large enough depending only on $n$ we have
\begin{equation}
\label{eq:boundalpha}
        \alpha_\gamma
        \le C_n b_\gamma^{-1}\mu_1((\alpha_\gamma+n-2)\gamma)
        \le \frac{C_n}{M\gamma^2} e^{-c\gamma} 
        \le e^{-c_n \gamma}.
\end{equation}

\medskip
\textbf{Step 6: conclusion.}
Let $\phi_\gamma\neq 0$ be a corresponding first eigenfunction, which is smooth on $\mb S^{n-1}$, and set
\[
        u_\gamma(r,\omega)\equiv r^{\alpha_\gamma}\phi_\gamma(\omega).
\]
Then $u_\gamma\in W^{1,2}(\mb B^n)$, since, recalling \eqref{eq:gradu}, we have
\[
        \int_{\mb B^n}|\D u_\gamma|^2
        =
        \left(\int_{\Sn}\alpha_\gamma^2\phi_\gamma^2+|\D_T\phi_\gamma|^2\d\sigma\right)
        \int_0^1 r^{2\alpha_\gamma+n-3}\d r<\infty.
\]
By Step 2 and equations \eqref{eq:PDEreduced} and \eqref{eq:eigenvalueidentity} we see that $u_\gamma$ solves $\ddiv(B_\gamma\D u_\gamma)=0$ classically in $\mb B^n\setminus\{0\}$. To check that $u_\gamma$ is a weak solution across the origin, fix $\zeta\in C_c^\infty(\mb B^n)$ and integrate by parts in $\mb B^n \setminus \e \mb B^n$:
$$
 \left|       \int_{\mb B^n\setminus \e \mb B^n} B_\gamma\D u_\gamma\cdot \D\zeta\right|=
        \left|\int_{\e \mb S^{n-1}}\zeta\, B_\gamma\D u_\gamma\cdot e_r\, \d\sigma \right|\leq  C_{\gamma,n} \|\zeta\|_{L^\infty}\e^{n+\alpha_\gamma-2} \to 0 \quad \text{ as } \e\to 0,
$$
since $ B_\gamma$ is bounded and $|\D u_\gamma(x)|\leq C_{n,\gamma} r^{\alpha_\gamma-1}.$
Thus $u_\gamma$ is a weak solution in the whole $\mb B^n$.

Applying Lemma \ref{lemma:sym} to $\widehat B_\gamma = \gamma^{-1}B_\gamma$ and $u_\gamma$, by Step 1 we obtain a symmetric coefficient $A_\gamma$ satisfying \eqref{eq:ellmain} with $K_\gamma= C_n \gamma$, and such that $u_\gamma$ solves \eqref{eq:PDEmain}. The conclusion now follows from \eqref{eq:boundalpha}.
\end{proof}

\appendix

\section{Proof of Theorem \ref{thm:GPT}}

In this section we explain how a non-trivial modification of the construction in Section \ref{sec:proof} leads to the proof of Theorem \ref{thm:GPT}. This modification was discovered essentially autonomously by ChatGPT-5.6 Pro, but the argument below has been completely rewritten and checked by the author.

\begin{proof}[Proof of Theorem \ref{thm:GPT}]
As before, we split the proof into several steps.

\medskip
\textbf{Step 1: ansatz.}
Let us take 
$$\phi_\gamma(\omega_n)\equiv \frac{\tanh(\gamma \omega_n)}{\tanh \gamma}, \qquad -1\leq \omega_n\leq 1,$$
thus $\phi_\gamma(\pm 1)=\pm 1$ and $\phi_\gamma(\omega_n)\to \sgn(\omega_n)$ as $\gamma\to \infty$, locally uniformly away from $\omega_n=0$.
Writing $x=r \omega$ as before, and with $s\equiv \log \frac 1 r $, we look for a solution of the form
\begin{equation}
\label{eq:ansatz2}
u(r\omega) \equiv U\bigg( \log \frac 1 r, \omega\bigg), \qquad
U(s,\omega)\equiv a(s)\phi_{\gamma(s)}(\omega_n), \qquad
\gamma(s) \equiv \log(S+s),
\end{equation}
for a large constant $S$ and a radial profile $a$ to be determined.
We take an operator of the form
\begin{equation}
\label{eq:defB2}
B(s,\omega) (\xi_r, \xi_T) = (M \gamma(s)^2 \xi_r + Y(s,\omega)\cdot \xi_T, \xi_T),
\end{equation}
exactly as in \eqref{eq:defB}, with $Y(s,\omega)$ tangent to the sphere, and $M$ very large to be chosen. Since
$$\D u(r \omega) = \frac 1 r (-U_s(s,\omega)\omega + \D_T U(s,\omega)), $$
we would like to write, exactly as in \eqref{eq:gradu},
\begin{equation}
\label{eq:desirableB}
B\,\D u = \frac 1 r (q(s,\omega) \omega + \D_T U(s,\omega)),
\end{equation}
for some scalar function $q$. Comparing \eqref{eq:defB2} and \eqref{eq:desirableB}, we see that we must have
\begin{equation}
\label{eq:eqY}
Y \cdot \D_T U = q + M\gamma^2 U_s.
\end{equation}
Note that we have
\begin{equation}
\label{eq:U-derivatives}
\D_TU=a(s)\phi'_{\gamma(s)}(\omega_n)\D_T\omega_n,
\qquad
U_s=a'(s)\phi_{\gamma(s)}(\omega_n)
+a(s)\gamma'(s)\partial_\gamma\phi_{\gamma(s)}(\omega_n),
\end{equation}
so $\D_TU$ vanishes at the poles. We choose $Y$ parallel
to $\D_T\omega_n$, and then \eqref{eq:eqY} forces
\begin{equation}
\label{eq:defY}
Y (s,\omega) \equiv \frac{q(s,\omega_n) + M \gamma(s)^2 U_s(s,\omega_n)}{ a(s) \phi'_{\gamma(s)}(\omega_n) (1-\omega_n^2)} \D_T \omega_n,
\end{equation}
since $|\D_T \omega_n|^2 = 1-\omega_n^2$.
In turn, by \eqref{eq:desirableB},
\begin{equation}
\label{eq:ODEq}
\ddiv(B\,\D u)=\frac1{r^2}\left(cq-q_s+\Delta_TU\right),
\qquad c\equiv n-2.
\end{equation}
Thus the PDE is equivalent to $q_s-cq=\Delta_TU$. This equation determines
$q$ from $a$, which we are yet to prescribe. Note that the homogeneous solution to this ODE is the fundamental solution
$$q_\tp{hom}(s) = e^{cs} = r^{-(n-2)}.$$
Hence, in order to avoid a Dirac mass at the origin, we require $e^{-c s } q(s,\cdot)\to 0$ as $s\to \infty$, which leads us to defining, since $\Delta_T U(s,\omega_n) = a(s) \Delta_T \phi_{\gamma(s)}(\omega_n)$, 
\begin{equation}
\label{eq:q2}
q(s,\omega_n) \equiv -\int_s^\infty e^{-c(t-s)} a(t) \Delta_T \phi_{\gamma(t)}(\omega_n) \,\d t.
\end{equation}

\medskip
\textbf{Step 2: finding the radial profile $a$.}
To compute $q(s,1)$, note that, for an axisymmetric function $f=f(\omega_n)$, we have
\begin{equation}
\label{eq:axisymmetric-laplacian}
\Delta_T f = (1-\omega_n^2) f''-(n-1) \omega_n f',
\end{equation}
by the chain rule, since
$|\D_T\omega_n|^2=1-\omega_n^2$ and
$\Delta_T\omega_n=-(n-1)\omega_n$. Thus\footnote{The quantity $\mu_\gamma$ should not be confused with the first positive eigenvalue
$\mu_1(t)$ used in Section \ref{sec:proof}. Here, $\mu_\gamma$ is not an
eigenvalue: it is the explicit value of $-\Delta_T\phi_\gamma$ at the poles. However, the two quantities play
analogous roles (hence the similar notation), since both are exponentially small angular quantities
that determine the radial profile in the construction.}
\begin{equation}
\label{eq:defmu}
\Delta_T \phi_\gamma(1) = -\mu_\gamma, \qquad \mu_\gamma \equiv (n-1)\phi'_\gamma(1) = \frac{2(n-1)\gamma}{\sinh(2 \gamma)}.
\end{equation}
It follows from \eqref{eq:q2} and \eqref{eq:defmu} that
\begin{equation}
\label{eq:q-pole}
q(s,1)=\int_s^\infty e^{-c(t-s)}a(t)\mu_{\gamma(t)}\,\d t>0.
\end{equation}
By \eqref{eq:U-derivatives} we have $\D_T U(s,1)=0$, since $\D_T\omega_n|_{\omega_n=1}=0$,  and
$U_s(s,1)=a'(s)$, since $\phi_\gamma(1)=1$ and $\p_\gamma \phi_\gamma(1)=0$. Hence the right-hand side in \eqref{eq:eqY} must
vanish at $\omega_n=1$, and we obtain
\begin{equation}
\label{eq:ODEa}
M \gamma(s)^2 a'(s) = -q(s,1).
\end{equation}
In order to have a discontinuity at the origin we want the radial profile to not vanish there, so we set $a(\infty)=1$. Therefore, integrating \eqref{eq:ODEa} from $s$ to $\infty$ and using \eqref{eq:q-pole}, we have
\begin{equation}
\label{eq:fixedpointa}
a(s) =1 + \frac 1 M \int_s^\infty \frac{\d t}{\gamma(t)^2} \int_t^\infty e^{-c(\tau-t)} a(\tau) \mu_{\gamma(\tau)} \d \tau\equiv 1 + (Ta)(s).
\end{equation}
We solve \eqref{eq:fixedpointa} in $L^\infty([0,\infty))$. Since
 $\gamma\mapsto\mu_\gamma$ is decreasing, for every
$f\in L^\infty([0,\infty))$ we have
\begin{equation}
\label{eq:T-bound}
\begin{aligned}
\|Tf\|_{L^\infty}
&\leq \frac{\|f\|_{L^\infty}}{Mc}
\int_0^\infty \frac{\mu_{\gamma(t)}}{\gamma(t)^2}\,\d t\leq \frac{C_n\|f\|_{L^\infty}}{M}
\int_0^\infty \frac{\d t}{(S+t)^2\log S}
= \frac{C_n\|f\|_{L^\infty}}{MS\log S}.
\end{aligned}
\end{equation}
Here, in the second inequality, we used the elementary bounds 
\begin{equation}
\label{eq:boundmu2}
\frac{\mu_{\gamma(t)}}{\gamma(t)^2}
\leq \frac{C_ne^{-2\gamma(t)}}{\gamma(t)}
= \frac{C_n}{(S+t)^2\log(S+t)}
\leq \frac{C_n}{(S+t)^2\log S}.
\end{equation}
Thus, for a fixed
$M$, we can choose $S=S(M)$ sufficiently large so that
$\|T\|\leq\frac12$. It follows that
\[
a=(\Id-T)^{-1}1=\sum_{j=0}^\infty T^j1
\]
is the unique bounded solution of \eqref{eq:fixedpointa}. Since $T$ is positive, $a\geq1$, while
$\|a\|_{L^\infty}\leq \frac{1}{1-\|T\|}\leq2.$
Similarly to \eqref{eq:T-bound}, using $a-1=Ta$ and retaining the lower
integration limit $s$, we obtain
\[
0\leq a(s)-1=(Ta)(s)
\leq \frac{2}{Mc}\int_s^\infty
\frac{\mu_{\gamma(t)}}{\gamma(t)^2}\,\d t
\to0
\qquad\text{as }s\to\infty.
\]
Thus $a$ is a bounded, $C^1$ solution of  \eqref{eq:ODEa} with
\begin{equation}
\label{eq:propsa}
1\leq a\leq 2, \qquad a'<0,\qquad \lim_{s\to \infty} a=1,
\end{equation}
where $a'<0$ follows from \eqref{eq:ODEa} and \eqref{eq:q-pole}.

\medskip
\textbf{Step 3: bounds on $Y$.}
We now state an elementary lemma, which we prove at the end of the appendix.
\begin{lemma}\label{lem:angular}
There is a constant $C_n$ such that, whenever $\eta\geq\gamma\geq2$ and
$\omega_n\in[-1,1]$,
\begin{align}
  |\partial_\gamma\phi_\gamma(\omega_n)|
  &\leq C_n(1-\omega_n^2)\phi_\gamma'(\omega_n),\label{eq:angular-1}\\
  |-\Delta_T\phi_\gamma(\omega_n)-\mu_\gamma\phi_\gamma(\omega_n)|
  &\leq C_n\gamma(1-\omega_n^2)\phi_\gamma'(\omega_n),\label{eq:angular-2}\\
  \phi_\eta'(\omega_n)&\leq\frac\eta\gamma\phi_\gamma'(\omega_n),
  \label{eq:angular-3}\\
  |\phi_\eta(\omega_n)-\phi_\gamma(\omega_n)|
  &\leq C_n\frac{\eta^2-\gamma^2}{\gamma}
  (1-\omega_n^2)\phi_\gamma'(\omega_n).\label{eq:angular-4}
\end{align}
\end{lemma}

A crucial point in the construction is to prove that $|Y|\lesssim\gamma$, which is the purpose of this step, and we use the lemma to estimate the numerator in its definition \eqref{eq:defY}. From
\eqref{eq:ODEa} and \eqref{eq:U-derivatives},
\begin{equation}
\label{eq:numerator-decomposition}
q(s,\omega_n)+M\gamma(s)^2U_s(s,\omega_n)
=q(s,\omega_n)-q(s,1)\phi_{\gamma(s)}(\omega_n)
+M\gamma(s)^2a(s)\gamma'(s)
\partial_\gamma\phi_{\gamma(s)}(\omega_n).
\end{equation}
For the first term on the right-hand side, \eqref{eq:q2} and \eqref{eq:q-pole} give
\begin{equation}
\begin{split}
q(s,\omega_n)-q(s,1)\phi_{\gamma(s)}(\omega_n)
={}\int_s^\infty e^{-c(t-s)}a(t)
\bigl[& -\Delta_T\phi_{\gamma(t)}(\omega_n)
-\mu_{\gamma(t)}\phi_{\gamma(t)}(\omega_n)\\
&\quad+\mu_{\gamma(t)}
\bigl(\phi_{\gamma(t)}(\omega_n)
-\phi_{\gamma(s)}(\omega_n)\bigr)\bigr]\,\d t.
\label{eq:q-difference2}
\end{split}
\end{equation}
We now estimate the right-hand side.
By \eqref{eq:ansatz2}, $0<\gamma'(\tau)=1/(S+\tau)\leq1$. Thus, for
$t\geq s$, writing $h=t-s$, we have
\begin{equation}
\label{eq:elementaryboundgamma}
\gamma(t)\leq\gamma(s)+h,
\qquad
\frac{\gamma(t)^2-\gamma(s)^2}{\gamma(s)}
\leq 2h+\frac{h^2}{\gamma(s)}.
\end{equation}
Now \eqref{eq:angular-2} and \eqref{eq:angular-3} give
\[
\left|-\Delta_T\phi_{\gamma(t)}(\omega_n)
-\mu_{\gamma(t)}\phi_{\gamma(t)}(\omega_n)\right|
\leq C_n\frac{\gamma(t)^2}{\gamma(s)}
(1-\omega_n^2)\phi_{\gamma(s)}'(\omega_n),
\]
while \eqref{eq:angular-4} and $\mu_{\gamma(t)}\leq C_n$ give
\[
\mu_{\gamma(t)}
\left|\phi_{\gamma(t)}(\omega_n)
-\phi_{\gamma(s)}(\omega_n)\right|
\leq C_n\frac{\gamma(t)^2-\gamma(s)^2}{\gamma(s)}
(1-\omega_n^2)\phi_{\gamma(s)}'(\omega_n).
\]
Inserting these estimates into \eqref{eq:q-difference2} and using
$1\leq a\leq2$ and \eqref{eq:elementaryboundgamma} yields
\begin{equation}
\begin{split}
\label{eq:q-difference-bound2}
\left|q(s,\omega_n)-q(s,1)\phi_{\gamma(s)}(\omega_n)\right|
& \leq C_n(1-\omega_n^2)\phi_{\gamma(s)}'(\omega_n)
\int_0^\infty e^{-ch}
\left(\frac{(\gamma(s)+h)^2}{\gamma(s)}
+2h+\frac{h^2}{\gamma(s)}\right)\,\d h\\
&\leq C_n\gamma(s)(1-\omega_n^2)
\phi_{\gamma(s)}'(\omega_n),
\end{split}
\end{equation}
 since $c=n-2>0$ and $\gamma(s)\geq2$.

We now fix $M$, depending only on $n$, sufficiently
large relative to the dimensional constants above, and then choose $S$
sufficiently large so that the requirement in Step 2 holds and
\begin{equation}
\label{eq:gamma-smallness2}
M\gamma(s)e^{-\gamma(s)}\leq1
\qquad\text{for every }s\geq0.
\end{equation}
Since $\gamma'=e^{-\gamma}$ and $a\leq 2$, \eqref{eq:angular-1} and
\eqref{eq:gamma-smallness2} give
\[
M\gamma(s)^2a(s)\gamma'(s)
\left|\partial_\gamma\phi_{\gamma(s)}(\omega_n)\right|
\leq C_n\gamma(s)(1-\omega_n^2)
\phi_{\gamma(s)}'(\omega_n).
\]
Combining this with \eqref{eq:numerator-decomposition} and
\eqref{eq:q-difference-bound2} yields
\begin{equation}
\label{eq:q-Us-bound2}
|q(s,\omega_n)+M\gamma(s)^2U_s(s,\omega_n)|
\leq C_n\gamma(s)(1-\omega_n^2)
\phi_{\gamma(s)}'(\omega_n)
\leq C_n\gamma(s)^2.
\end{equation}
Since $|\D_T\omega_n|=\sqrt{1-\omega_n^2}$, the definition
\eqref{eq:defY} and $a\geq1$ imply
\begin{equation}
\label{eq:Y-bound2}
|Y(s,\omega)|\leq C_n\gamma(s)\sqrt{1-\omega_n^2}
\leq C_n\gamma(s).
\end{equation}
In particular, $Y(s,\cdot)$ extends continuously to the poles.
We note that, for much simpler reasons, the same bound holds for the tangential field $Y$ constructed in Section \ref{sec:proof}.

\medskip
\textbf{Step 4: ellipticity.} The bound \eqref{eq:Y-bound2} makes the
off-diagonal term in \eqref{eq:defB2} negligible compared with
$M\gamma^2$. Indeed, choosing the constant $M=M(n)$ sufficiently
large, \eqref{eq:Y-bound2} gives
$|Y|^2/(M\gamma^2)\leq\frac12$. Young's inequality then gives,
exactly as in Step~1 of Section~\ref{sec:proof},
\begin{equation}
\label{eq:B-lower2}
B\xi\cdot\xi
\geq\frac12\left(M\gamma^2\xi_r^2+|\xi_T|^2\right)
\geq\frac12|\xi|^2,
\qquad
B^{-1}\xi\cdot\xi
\geq\frac1{2M\gamma^2}\xi_r^2+\frac12|\xi_T|^2
\geq\frac1{2M\gamma^2}|\xi|^2.
\end{equation}
Therefore, by applying exactly the same pointwise construction used in the proof of Lemma \ref{lemma:sym}, we find a measurable symmetric matrix
field $A$ such that $A\,\D u=4B\,\D u$ and, after increasing $S$ if
necessary, $A$ satisfies \eqref{eq:elldeg} with
\[
\Lambda(x)=16M\gamma(s)^2.
\]
%Since
%$\d x=e^{-ns}\,\d s\,\d\sigma$, we also have
%\begin{equation}
%\label{eq:Lambda-integrability2}
%\int_{\mb B^n}\exp\bigl(\Lambda(x)\bigr)\,\d x
%\leq C_n\int_0^\infty
%\exp\left(-ns+16M\log^2(S+s)\right)\,\d s<\infty,
%\end{equation}
%which shows \eqref{eq:expint}.
Now \eqref{eq:Lambda} follows immediately.

\medskip
\textbf{Step 5: membership in $H^1(\mb B^n,A)$.} Since $a\leq 2$,  we first note the estimates
\begin{equation}
\label{eq:finiteenergy}
|u|\leq2,\qquad
|\D_TU|\leq C_n\gamma,\qquad
|U_s|\leq \frac{C_n}{M},\qquad
 |q|\leq C_n\gamma^2.
\end{equation}
Indeed, the first estimate follows from \eqref{eq:ansatz2} and
$|\phi_\gamma|\leq 1$, and the second from \eqref{eq:U-derivatives} and
$\phi'_\gamma\leq C\gamma$. For the last inequality, first note that by \eqref{eq:q-pole}, the monotonicity
of $\mu_\gamma$ and $\mu_\gamma\leq C_n\gamma^2$, cf.\ \eqref{eq:boundmu2}, we have \(q(s,1)\le \frac{2}{c}\mu_{\gamma(s)}
\le C_n\gamma(s)^2\). Hence \eqref{eq:q-difference-bound2}, together with
$\phi'_\gamma\leq C\gamma$, give $|q|\leq C_n\gamma^2$. Finally,
\eqref{eq:q-Us-bound2}, together with the bound for $q$, gives
$|U_s|\leq C_n/M$.

Moreover, \eqref{eq:Y-bound2} applied to \eqref{eq:defB2} gives
\[
B(\xi_r,\xi_T)\cdot(\xi_r,\xi_T)
\leq C_n\left(M\gamma^2\xi_r^2+|\xi_T|^2\right).
\]
Applying this with
$(\xi_r,\xi_T)=r^{-1}(-U_s,\D_TU)$ gives
\[
B\,\D u\cdot\D u
\leq\frac{C_n}{r^2}
\left(M\gamma^2|U_s|^2+|\D_TU|^2\right).
\]
Since $A\,\D u=4B\,\D u$, \eqref{eq:finiteenergy} and spherical
coordinates yield
\[
\int_{\mb B^n}\bigl(\Lambda u^2+A\D u\cdot\D u\bigr)\,\d x
\leq C_n\int_0^\infty
\left(e^{-ns}+e^{-cs}\right)\gamma(s)^2\,\d s<\infty.
\]
To verify membership in the completion defining $H^1(\mb B^n,A)$, let
$\chi_R\in C^\infty([0,\infty))$ equal $1$ on $[0,R]$ and $0$ on
$[R+1,\infty)$, with $|\chi_R'|\leq C$, and set
$u_R(r\omega)=\chi_R(s)u(r\omega)$. Then
$u_R\in C^1(\mb B^n)$. Writing $w_R=u-u_R$, since $A$ is symmetric
and positive definite,
\[
A\D w_R\cdot\D w_R
\leq2(1-\chi_R)^2A\D u\cdot\D u
+2A(u\D\chi_R)\cdot(u\D\chi_R).
\]
The first term on the right, as well as $\Lambda w_R^2$, converges to zero in $L^1$
by the preceding estimate, while
\[
\int_{\mb B^n}A(u\D\chi_R)\cdot(u\D\chi_R)\,\d x
\leq C_n\int_R^{R+1}e^{-cs}\gamma(s)^2\,\d s\to 0,
\]
so $u_R\to u$ in $H^1(\mb B^n,A)$.

\medskip
\textbf{Step 6: conclusion.} By \eqref{eq:ODEq} and \eqref{eq:q2},
$\ddiv(B\,\D u)=0$ away from the origin. Also,
\eqref{eq:desirableB} and $A\,\D u=4B\,\D u$ give
$(A\D u)(r\omega)\cdot\omega=4q(s,\omega)/r$. Hence, for
$\varphi\in C_c^\infty(\mb B^n)$, \eqref{eq:finiteenergy} and
integration by parts outside $\mb B_r^n$ give
\[
\left|\int_{\partial\mb B_r^n}\varphi A\D u\cdot\omega\,\d\sigma\right|
\leq C_n\|\varphi\|_{L^\infty}r^{n-2}
\log^2\left(S+\log\frac1r\right)\to0
\]
since $n\geq3$.
Thus $\int_{\mb B^n}A\D u\cdot\D\varphi\,\d x=0$ and by density this holds
for every $\varphi\in H^1_0(\mb B^n,A)$.

Finally, by \eqref{eq:propsa}, $a(s)\to1$, while
$\phi_{\gamma(s)}(\omega_n)\to\sgn(\omega_n)$ locally uniformly away
from the equator, as $s\to \infty$. Consequently $u(r\omega)\to1$ on every cone
$\{\omega_n\geq\delta\}$, and $u(r\omega)\to-1$ on every cone
$\{\omega_n\leq-\delta\}$, where $\delta>0$. It follows that no representative of $u$ is continuous at the origin. 
\end{proof}

\begin{proof}[Proof of Lemma \ref{lem:angular}]
Direct differentiation gives
\begin{equation}\label{eq:parameter-derivative}
  \partial_\gamma\phi_\gamma(\omega_n)
  =\frac1\gamma\phi_\gamma'(\omega_n)
  \left(\omega_n-\frac{\sinh(2\gamma \omega_n)}{\sinh(2\gamma)}\right).
\end{equation}
Also, using \eqref{eq:axisymmetric-laplacian} and
$\phi_\gamma''=-2\gamma\tanh(\gamma \omega_n)\phi_\gamma'$, one obtains
\begin{align}
  -\Delta_T\phi_\gamma-\mu_\gamma\phi_\gamma
  ={}&2\gamma(1-\omega_n^2)\tanh(\gamma \omega_n)\phi_\gamma'
  +(n-1)\phi_\gamma'
  \left(\omega_n-\frac{\sinh(2\gamma \omega_n)}{\sinh(2\gamma)}\right).
  \label{eq:laplacian-identity}
\end{align}
Here we used the exact identity, cf.\ \eqref{eq:defmu},
\[
  \mu_\gamma\phi_\gamma(\omega_n)
  =(n-1)\phi_\gamma'(\omega_n)
  \frac{\sinh(2\gamma \omega_n)}{\sinh(2\gamma)}.
\]
We claim that
\begin{equation}\label{eq:bracket-bound}
 \left|F_\gamma(\omega_n)\right|\equiv  \left|\omega_n-\frac{\sinh(2\gamma \omega_n)}{\sinh(2\gamma)}\right|
  \leq C\gamma(1-\omega_n^2).
\end{equation}
Both sides are even, so it suffices to
consider $0\leq\omega_n\leq1$, and the claim follows from $F_\gamma(1)=0$ and 
$|F'_\gamma|\leq 1+2\gamma\coth(2\gamma)\leq C\gamma$. Equations
\eqref{eq:parameter-derivative}--\eqref{eq:bracket-bound} prove
\eqref{eq:angular-1} and \eqref{eq:angular-2}. Furthermore, for $\eta \geq \gamma$,
\[
  \frac{\phi_\eta'(\omega_n)}{\phi_\gamma'(\omega_n)}
  =\frac\eta\gamma\frac{\tanh\gamma}{\tanh\eta}
  \frac{\cosh^2(\gamma \omega_n)}{\cosh^2(\eta \omega_n)}
  \leq\frac\eta\gamma,
\]
which proves \eqref{eq:angular-3}. Finally, integrate
\eqref{eq:angular-1} with respect to the parameter and use
\eqref{eq:angular-3}:
\[
  |\phi_\eta(\omega_n)-\phi_\gamma(\omega_n)|
  \leq C_n(1-\omega_n^2)\phi_\gamma'(\omega_n)
  \int_\gamma^\eta\frac\kappa\gamma\,\d\kappa
  \leq C_n\frac{\eta^2-\gamma^2}{\gamma}
  (1-\omega_n^2)\phi_\gamma'(\omega_n).
\]
This is \eqref{eq:angular-4}.
\end{proof}

We conclude the paper by briefly commenting on how \cite{Armstrong2026a} can be used to see that condition \eqref{eq:doubleexp} yields continuity of solutions; this was communicated to the author by one of the referees. In the next remark we follow the notation in \cite{Armstrong2026a}.

\begin{remark}\label{rem:coarse}
Fix H\"older conjugate exponents $p,q\in(1,\infty)$ and a cube
$Q_0\Subset\mb B^n$.
The limiting case where $p=\infty, q=1$ or vice-versa follows from a similar argument.
By  \cite[Lemma 2.4]{Armstrong2026a}, 
for every cube $R\subset Q_0$ and every admissible cone $\mathcal C$, we have
\[
 \|A(R;\mathcal C)\|\leq\fint_R\Lambda\,\d x,
 \qquad \|A_*^{-1}(R)\|\leq\fint_R\lambda^{-1}\,\d x.
\]
Jensen's inequality, after increasing constants to take care of small
values of $\lambda^{-1}$ and $\Lambda$, gives
\[
\begin{aligned}
 \left(\fint_R\lambda^{-1}\,\d x\right)^{1/2}
 &\leq C\left(1+\left[c^{-1}\log\log\left(
 \fint_R\exp \exp(c\lambda^{-p/2})\,\d x\right)\right]^{1/p}\right),\\
 \left(\fint_R\Lambda\,\d x\right)^{1/2}
 &\leq C\left(1+\left[c^{-1}\log\log\left(
 \fint_R\exp \exp(c\Lambda^{q/2})\,\d x\right)\right]^{1/q}\right).
\end{aligned}
\]
Since the last integrals are bounded by the fixed local integrals over
$Q_0$, this yields, for descendants $R$ of relative depth $h$ inside a
cube $Q_j$ of generation $j$,
\[
\begin{aligned}
 \left(\fint_R\lambda^{-1}\,\d x\right)^{1/2}
 &\leq C_0+C\left(\frac{\log(j+h+2)}c\right)^{1/p},\\
 \left(\fint_R\Lambda\,\d x\right)^{1/2}
 &\leq C_0+C\left(\frac{\log(j+h+2)}c\right)^{1/q}.
\end{aligned}
\]
Inserting these estimates into the definitions of $\lambda_t$ and
$\Lambda_s$ from \cite{Armstrong2026a} gives, for fixed $s,t\in(0,1)$ with $s+t<1$, and for a constant $C$ independent of $j$ and $c$,
\[
 \lambda_t(Q_j)^{-1/2}
 \leq C_0+C\left(\frac{\log(j+2)}c\right)^{1/p},
 \qquad
 \Lambda_s(Q_j;\mathcal C)^{1/2}
 \leq C_0+C\left(\frac{\log(j+2)}c\right)^{1/q}.
\]
Since $1/p+1/q=1$, it follows that
\[
 \Theta_{s,t}(Q_j;\mathcal C)^{1/2}
 \equiv \Lambda_s(Q_j;\mathcal C)^{1/2}\lambda_t(Q_j)^{-1/2}
 \leq C_0+\frac Cc\log(j+2).
\]
Hence Theorem 1.2 of \cite{Armstrong2026a} gives a Harnack constant
$H_j\leq C_0(j+2)^{C/c}$ on $Q_j$.

Choose nested cubes with $Q_{j+1}\subset\frac18Q_j$. The local
boundedness theorem \cite[Theorem 1.1]{Armstrong2026a} makes
$m_j=\operatorname*{ess\,inf}_{Q_j}u$ and
$M_j=\operatorname*{ess\,sup}_{Q_j}u$ finite. Applying Harnack to
$u-m_j$ and $M_j-u$ gives
\[
 \mathop{\rm osc}_{Q_{j+1}}u
 \leq\frac{H_j-1}{H_j+1}\mathop{\rm osc}_{Q_j}u.
\]
Thus the oscillation tends to zero whenever $\sum_jH_j^{-1}=\infty$, which happens if $c$ is sufficiently large. Since $Q_0$ was arbitrary, $u$ has a
continuous representative. 
\end{remark}

%The following code condenses the bibliography
\let\oldthebibliography\thebibliography
\let\endoldthebibliography\endthebibliography
\renewenvironment{thebibliography}[1]{
  \begin{oldthebibliography}{#1}
    \setlength{\itemsep}{0.5pt}
    \setlength{\parskip}{0.5pt}
}
{
  \end{oldthebibliography}
}

	{\small
	\bibliographystyle{abbrv-andre}
	\bibliography{/Users/andreguerra/Library/CloudStorage/Dropbox/Projects/library.bib,DGNM-exponent-local}
	}

\end{document}